\documentclass[12pt, a4paper]{article}

\usepackage{mathrsfs}

\usepackage{srcltx}
\usepackage{amsfonts,amssymb,mathrsfs,amsmath,cmtiup}
\usepackage{amsthm, eucal, eufrak}

\newtheorem{theorem}{Theorem}[section]
\newtheorem{lemma}[theorem]{Lemma}
\newtheorem{proposition}[theorem]{Proposition}

\theoremstyle{definition}
\newtheorem*{definition}{Definition}
\newtheorem{remark}[theorem]{Remark}
\newtheorem*{Index Convention}{Index Convention}
\newtheorem*{notation}{Notation}
\newtheorem{example}[theorem]{Example}

\def\keywords#1{\par\medskip
\noindent\textbf{Key words.} #1}

\def\subjclass#1{{\renewcommand{\thefootnote}{}
\footnote{\emph{Mathematics Subject Classification (2010):} #1}}}

\begin{document}
\let\le=\leqslant
\let\ge=\geqslant
\let\leq=\leqslant
\let\geq=\geqslant

\title{Finite $p$-groups with a  Frobenius group of automorphisms whose kernel is a cyclic $p$-group}

\markright{}

\author{{E.\,I.~Khukhro}\footnote{The first author was supported by the Russian Science Foundation, project no. 14-21-00065.}\\ \small Sobolev Institute of Mathematics, Novosibirsk,
630\,090, Russia\\[-1ex] \small khukhro@yahoo.co.uk \\
{N.\,Yu.~Makarenko}\footnote{The second author was supported by the Russian Foundation
for Basic Research, project no. 13-01-00505}\\
\small Universit\'{e} de Haute Alsace, Mulhouse, 68093, France
and\\ \small Sobolev Institute of
Mathematics, Novosibirsk, 630\,090, Russia\\
[-1ex] \small  natalia\_makarenko@yahoo.fr}

\date{}
\maketitle

\begin{center}{\it to
Victor Danilovich Mazurov on the occasion of his 70th birthday}
\end{center}

\begin{abstract}
Suppose that a finite $p$-group $P$
admits a Frobenius group of
automorphisms $FH$ with kernel $F$ that is a cyclic $p$-group and
with complement $H$.   It is proved that if the fixed-point
subgroup $C_P(H)$ of the complement is nilpotent of class $c$,
then $P$ has a characteristic subgroup of index bounded in terms
of $c$, $|C_P(F)|$, and $|F|$ whose nilpotency class is bounded in
terms of $c$ and $|H|$ only. Examples show that the condition of
$F$ being cyclic is essential. The proof is based on a Lie ring
method and a theorem of the authors and P.~Shumyatsky about Lie
rings with a metacyclic Frobenius group of automorphisms $FH$. It
is also proved that $P$ has a characteristic subgroup of
$(|C_P(F)|, |F|)$-bounded index whose order and rank are bounded
in terms of $|H|$ and the order and rank of $C_P(H)$,
respectively, and whose exponent is bounded in terms of the
exponent of $C_P(H)$.
\end{abstract}


\section{Introduction}

It has long been known that results on `semisimple'
fixed-point-free automorphisms of nilpotent groups and Lie rings
can be applied for studying  `unipotent' $p$-automorphisms of
finite $p$-groups. Alperin \cite{al} was the first to use Higman's
theorem on Lie rings and nilpotent groups with a fixed-point-free
automorphism of prime order $p$ in the study of a finite $p$-group
$P$ with an automorphism $\varphi $ of order $p$. Namely, Alperin
\cite{al} proved that the derived length of $P$ is bounded in
terms of the number of fixed points $p^m=|C_P(\varphi  )|$. Later the
first author \cite{khu85} improved the argument to obtain a
subgroup of $P$ of $(p,m)$-bounded index and of $p$-bounded
nilpotency class, and the second author \cite{mak88} noted  that
this class can be bounded by $h(p)$, where $h(p)$ is Higman's
function bounding the nilpotency class of a Lie ring  or a
nilpotent group with a fixed-point-free automorphism  of order
$p$. Henceforth we write for brevity, say, ``$(a,b,\dots
)$-bounded'' for ``bounded above by some function depending only
on $a, b,\dots $''. Further strong results on $p$-automorphisms of finite $p$-groups
were obtained by Kiming \cite{kim}, McKay \cite{mck}, Shalev  \cite{sha}, Khukhro \cite{khu93},
Medvedev \cite{med98, med99}, Jaikin-Zapirain \cite{jai}, Shalev and Zelmanov \cite{sha-zel}
giving subgroups of bounded index and of bounded derived length or
nilpotency class. The proofs of most of these `unipotent' results were
also based on the `semisimple' theorems of Higman \cite{hi}, Kreknin
\cite{kr}, Kreknin and Kostrikin \cite{kr-ko} on
fixed-point-free automorphisms of Lie rings.

In the present paper `unipotent' theorems are derived from the
recent `semisimple'  results of the authors and Shumyatsky
\cite{khu-ma-shu, khu-ma-shu-DAN} about groups $G$ (and Lie
rings $L$) admitting a Frobenius group $FH$ of automorphisms with
kernel $F$ and complement $H$. The results concern the connection
between the nilpotency class, order, rank, and exponent of $G$ and
the corresponding parameters of $C_G(H)$. The more difficult of
these results is about the nilpotency class, and its proof is
based on the corresponding Lie ring theorem. Namely, it was proved
in \cite{khu-ma-shu} that if the kernel $F$ is cyclic and acts on
a Lie ring $L$ fixed-point-freely, $C_L(F)=0$, and the fixed-point
subring $C_L(H)$ of the complement is nilpotent of class $c$, then
$L$ is nilpotent of $(c,|H|)$-bounded class (under certain
assumptions on the additive group of $L$, which are satisfied in
many important cases, like $L$ being an algebra over a field, or
being finite). Note that examples show that the condition of $F$
being cyclic is essential. This Lie ring result also implied a
similar result for a finite group $G$ with a Frobenius group $FH$
of automorphisms with cyclic fixed-point-free kernel $F$ such that
$C_G(H)$ is nilpotent of class $c$, with reduction to nilpotent
case provided by classification and representation theory
arguments. The fixed-point-free action of $F$ alone was known to
imply nice properties of the Lie ring (solubility of $|F|$-bounded
derived length by Kreknin's theorem \cite{kr}) and of the group
(solubility and well-known bounds for the Fitting height due to
 Thompson \cite{th2}, Kurzweil \cite{kurz},  Turull \cite{tu},  and others --- although  an analogue of
Kreknin's theorem is still an open problem for groups). But the conclusions of the results in
\cite{khu-ma-shu} are in a sense much stronger, due to the
combination of the hypotheses on fixed points of $F$ and $H$, either of which on its own is  insufficient.

We now state the `unipotent' version of the nilpotency class result in \cite{khu-ma-shu}.

 \begin{theorem}\label{t1}
 Suppose that a finite $p$-group $P$ admits a
Frobenius group $FH$  of automorphisms with cyclic kernel $F$ of
order $p^k$.
Let $c$ be the nilpotency class of the fixed-point subgroup $C_P(H)$ of the complement.
Then $P$ has a characteristic
subgroup of index bounded in terms of  $c$, $|F|$, and $|C_P(F)|$ whose nilpotency class is bounded in
terms of $c$ and $|H|$ only.
 \end{theorem}

The proof is quite similar to the proofs of the
aforementioned results of Alperin \cite{al} and Khukhro
\cite{khu85}, with the Lie ring theorem in \cite{khu-ma-shu}
taking over the role of the Higman--Kreknin--Kostrikin theorem.
However, first a certain combinatorial corollary of that Lie ring
theorem has to be derived (Proposition~\ref{combinatorial}). Example~\ref{example} shows that the condition
of the kernel $F$ being cyclic in Theorem~\ref{t1} is essential.

We now state the unipotent versions of the rank, order, and exponent results in
\cite{khu-ma-shu}. (By the rank we mean the minimum number $r$ such that every subgroup can be generated by $r$ elements.)

\begin{theorem} \label{t-g}
 Suppose that a finite $p$-group $P$ admits a
Frobenius group $FH$  of automorphisms with cyclic kernel $F$ of
order $p^k$.
Then $P$ has a characteristic subgroup $Q$ of index bounded in terms of $|F|$ and $|C_P(F)|$ such that

{\rm (a)} the order of $Q$ is at most $|C_P(H)|^{|H|}$;

{\rm (b)} the rank of $Q$ is at most $r|H|$, where  $r$ is the rank of $C_P(H)$;

{\rm (c)} the exponent of $Q$ is at most $p^{2e}$, where $p^e$ is the exponent of $C_P(H)$.
\end{theorem}

Note that the estimates for the order and rank are best-possible,
and for the exponent close to being best-possible (and independent
of $|FH|$). The proof is facilitated by a straightforward
reduction to powerful $p$-groups. Then certain
versions of
the `free $H$-module arguments' are applied to abelian
$FH$-invariant sections. If a finite group $G$ admits a
Frobenius group of automorphisms $FH$ with complement $H$ and with
kernel $F$ acting fixed-point-freely, then every elementary
abelian $FH$-invariant section of $G$ is a free $kH$-module (for
various prime fields $k$). This is exactly what provides a
motivation for seeking results bounding various parameters of $G$
in terms of those of $C_P(H)$ and $|H|$.  In the
`semisimple' situation this fact is a basis of the results on the
order and rank in \cite{khu-ma-shu}. The exponent result in
\cite{khu-ma-shu} is more difficult, but in our unipotent
situation a simpler argument can be used based on powerful
$p$-groups to produce a much better result, with the estimate for
the exponent depending only on the exponent of $C_P(H)$.

It should be mentioned that the `semisimple' results on the order
and rank in \cite{khu-ma-shu} do not assume the kernel to be
cyclic. What the `unipotent' analogue of these results for non-cyclic kernel %!!
should be is unclear at the
moment. The results of the present paper can be regarded as
generalizations of the results of \cite{khu-ma-shu}, where the
kernel $F$ acts on $G$ fixed-point-freely, to the case of `almost
fixed-point-free' kernel. It is natural to expect that similar
restrictions, in terms of the complement $H$ and its fixed points
$C_G(H)$, should hold for a subgroup of index bounded in terms of
$|C_G(F)|$ and other parameters: `almost fixed-point-free' action
of $F$ implying that $G$ is `almost' as good as when $F$ acts
fixed-point-freely. In the coprime `semisimple' situation such
restrictions were recently obtained in \cite{khu13} for the order
and rank of $G$, and in \cite{khu-mak13} and \cite{mak-khu13} for
the nilpotency class. For the moment it is unclear how to combine
these semisimple and unipotent results in a general setting,
without assumptions on the orders of $G$ and $FH$; note that the
results in \cite{khu-ma-shu}
for the fixed-point-free kernel
were free of such assumptions.

The authors thank the referee for careful reading and helpful comments. %!!??

\section{Lie ring technique}\label{s-l}

First we recall some definitions and notation.
Products in a Lie ring are called commutators.
The Lie subring generated by a subset~$S$ is denoted by $\langle S\rangle $  and the ideal by
${}_{{\rm id}}\!\left< S \right>$.

Terms of the lower central series of a Lie ring $L$
are defined by induction: $\gamma_1(L)=L$; $\gamma_{i+1}(L)=[\gamma_i(L),L].$
By definition a Lie ring $L$ is nilpotent of class~$h$ if
$\gamma_{h+1}(L)=0$.

A simple commutator $[a_1,a_2,\dots ,a_s]$ of weight (length)
 $s$ is by definition the commutator $[...[[a_1,a_2],a_3],\dots ,a_s]$.

Let $A$ be an additively written abelian group. A Lie ring $L$
is \textit{$A$-graded} if
$$L=\bigoplus_{a\in A}L_a\qquad \text{ and }\qquad[L_a,L_b]\subseteq L_{a+b},\quad a,b\in A,$$
where the grading components $L_a$ are additive subgroups of $L$. Elements of the $L_a$
are called \textit{homogeneous} (with respect to this grading), and commutators in homogeneous
elements \textit{homogeneous commutators}. An additive subgroup
 $H$ of $L$ is said to be \textit{homogeneous}
if $H=\bigoplus_a (H\cap L_a)$; then we set $H_a=H\cap L_a$.
Obviously, any subring or an ideal generated by homogeneous
additive subgroups is
 homogeneous. A homogeneous subring and the
quotient ring by a homogeneous ideal can be regarded as
$A$-graded rings with the induced gradings.

Suppose that a Lie ring $L$ admits a Frobenius
group of automorphisms $FH$ with cyclic kernel $F=\langle\varphi \rangle $ of order $n$.
Let $\omega$ be a primitive $n$-th root of unity. We extend the
ground ring by $\omega$ and denote by $\widetilde L$ the ring
$L\otimes _{{\Bbb Z} }{\Bbb Z} [\omega ]$. Then $\varphi $
naturally acts on $\widetilde L$ and, in particular,
$C_{\widetilde L}(\varphi  ) = C_L(\varphi )\otimes  _{{\Bbb Z} }{\Bbb Z} [\omega ]$.

\begin{definition} We define $\varphi$-\textit{components} $L_k$ for
$k=0,\,1,\,\ldots ,n-1$ as the `eigensubspaces'
$$L_k=\big\{ a\in \widetilde L\mid a^{\varphi}=\omega ^{k}a\big\} .$$
\end{definition}

It is well known that $n\widetilde L\subseteq  L_0 + L_1 + \dots +
L_{n-1}$ (see, for example,~\cite[Ch.~10]{hpbl}). This
decomposition resembles a $({\Bbb Z}/n{\Bbb Z})$-grading because
of the inclusions $[L_s,\, L_t]\subseteq L_{s+t\,({\rm mod}\,n)}$,
but the sum of $\varphi$-components is not direct in general.

\begin{definition}
We refer to commutators in  elements of $\varphi$-components as being \textit{$\varphi$-homo\-ge\-neous}.
\end{definition}

\begin{Index Convention} Henceforth a small letter
with index $i$ denotes an element of the $\varphi$-component $L_i$, so that the index only indicates the
$\varphi$-component to which this
element belongs: $x_i\in L_i$. To lighten the notation we will not
use numbering indices for elements in $L_j$, so that
different elements can be denoted by the same symbol when
it only matters to which $\varphi$-component these elements belong. For example, $x_1$ and $x_1$ can be
different elements of $L_1$, so that $[x_1,\, x_1]$ can be a nonzero element of
$L_2$. These indices will be considered modulo~$n$; for example, $a_{-i}\in
L_{-i}=L_{n-i}$.
\end{Index Convention}

Note that under the Index Convention
a $\varphi$-homo\-ge\-neous commutator
 belongs to the $\varphi$-component $L_s$, where
 $s$ is the sum modulo $n$ of the indices of all the elements occurring in this
commutator.

Since the kernel $F$ of the Frobenius group $FH$ is cyclic,
the complement $H$ is also cyclic. Let $H= \langle h
\rangle$ be of order $q$ and $\varphi^{h^{-1}} = \varphi^{r}$ for some $1\leq
r \leq n-1$. Then $r$ is a primitive $q$-th root of unity in the ring ${\Bbb Z}/n {\Bbb Z}$.

The group $H$ permutes the $\varphi$-components $L_i$ as follows:
$L_i^h = L_{ri}$ for all $i\in \Bbb Z/n\Bbb Z$.
Indeed, if $x_i\in L_i$, then $(x_i^{h})^{\varphi} =
x_i^{h\varphi h^{-1}h} = (x_i^{\varphi^{r}})^h =\omega^{ir}x_i^h$, so that
$L_i^h\subseteq L_{ir}$; the reverse inclusion is obtained by applying the same argument to $h^{-1}$.

\begin{notation} In what follows, for a given $u_k\in L_k$ we denote the element
$u_k^{h^i}$ by $u_{r^ik}$ under the Index Convention,
since $L_k^{h^i} = L_{r^ik}$. We denote the $H$-orbit of an element $x_i$ by
$O(x_{i})=\{x_{i},\,\, x_{ri},\dots, x_{r^{q-1}i}\}$.
\end{notation}

 We are going to prove a combinatorial consequence of
the Makarenko--Khukhro--Shumyatsky theorem in~\cite{khu-ma-shu},
which we state in a somewhat different form, in terms of $({\Bbb Z} /n{\Bbb Z}
)$-graded Lie rings with a cyclic group of automorphisms $H$.

\begin{theorem} [{\cite[Theorem 5.5~(b)]{khu-ma-shu}}]\label{kh-ma-shu10-1}
Let $M=\bigoplus _{i=0}^{n-1} M_i$ be a $({\Bbb Z} /n{\Bbb Z} )$-graded Lie
ring with grading components $M_i$ that are additive subgroups
satisfying the inclusions $[M_i,M_j]\subseteq M_{i+j\,({\rm
mod}\,n)}$. Suppose  $M$ admits a finite cyclic group of
automorphisms $H=\langle h\rangle$ of order $q$ such that
$M_i^h=M_{ri}$ for some element $r\in {\Bbb Z}/n{\Bbb Z}$ having
multiplicative order~$q$. If $M_0=0$ and $C_M(H)$ is nilpotent of class $c$,
then for some functions $u=u(c,q)$ and $f=f(c,q)$ depending only
on $c$ and $q$, the Lie subring $n^{u}M$ is nilpotent of class
$f-1$, that is, $\gamma_{f}(n^{u}M)=n^{uf}\gamma _{f}(M)=0$.
\end{theorem}

The corresponding theorems  in \cite{khu-ma-shu} were stated about
Lie rings admitting a Frobenius group $FH$ of automorphisms with
cyclic kernel $F=\langle\varphi \rangle$ of order $n$. After extension
of the ground ring,  the $\varphi $-components behave like components of
a $({\Bbb Z} /n{\Bbb Z} )$-grading, as we saw above. In fact, the proofs in
\cite{khu-ma-shu} only used the `grading' properties of the
$\varphi $-components, so that Theorem~\ref{kh-ma-shu10-1} was actually
proved therein.
The following proposition is a combinatorial consequence of this theorem.

\begin{proposition}\label{combinatorial}
Let $f=f(c,q)$, $u=u(c,q)$ be the functions in
Theorem~\ref{kh-ma-shu10-1}. Suppose that a Lie ring $L$ admits a
Frobenius group of automorphisms $FH$ with cyclic kernel
$F=\langle \varphi  \rangle$ of order $n$ and with complement $H$ of
order $q$ such that the fixed-point subring $C_L(H)$ of the
complement is nilpotent of class $c$.
 Then for the $(c,q)$-bounded number $w=(u+1)f$ %!!pust'
the $n^w$-th multiple
 $n^w [x_{i_1}, x_{i_2},\ldots, x_{i_{f}}]$ of every simple
$\varphi$-homo\-ge\-neous
commutator in $\widetilde L= L\otimes _{{\Bbb Z} }{\Bbb Z} [\omega ]$
of weight $f$ with non-zero indices can be
represented as a linear combination of $\varphi$-homo\-ge\-neous
commutators of the same weight $f$ in elements of the union of $H$-orbits
$\bigcup_{s=1}^f O(x_{i_s})$ each of which contains a subcommutator
with zero sum of indices modulo $n$.
\end{proposition}

\begin{remark} Similar combinatorial propositions
were also proved for Lie algebras in \cite{mak-khu13} and for Lie rings whose
 ground ring contains the inverse
 of $n$ in \cite{khu-mak13}.
\end{remark}

\begin{proof}
The idea of the proof is application of
Theorem~\ref{kh-ma-shu10-1} to a free Lie ring with operators
$FH$. Given arbitrary (not necessarily
distinct) non-zero
 elements $i_1, i_2,\dots, i_f\in \Bbb Z /n{\Bbb Z}$, we
consider a free Lie ring $K$ with $qf$ free generators in
the set
$$Y=\{\underbrace{y_{i_1}, y_{ri_1}, \ldots, y_{r^{q-1}i_1}}_{O(y_{i_1})},\,\,\,
\underbrace{y_{i_2}, y_{ri_2}, \ldots,
y_{r^{q-1}i_2}}_{O(y_{i_2})},\ldots,
\underbrace{y_{i_{f}},y_{ri_f}, \ldots,
y_{r^{q-1}i_f}}_{O(y_{i_f})}\},$$
where indices  are formally assigned and regarded modulo $n$ and the subsets
$O(y_{i_s})=\{y_{i_s}, y_{ri_s}, \ldots, y_{r^{q-1}i_s}\}$ are disjoint.
Here, as in the Index Convention, we do not use numbering indices, that is,
all elements $y_{r^ki_j}$ are by definition different free generators, even if
indices coincide. (The Index Convention will come into force in a moment.) For
every $i=0,\,1,\,\dots ,n-1$ we define
the additive subgroup $K_i$ generated by all
commutators in the generators $y_{j_s}$ in which the sum of
indices of all entries is equal to $i$ modulo $n$. Then
$K=K_0\oplus K_1\oplus \cdots \oplus K_{n-1}$. It is also obvious that
$ [K_i,K_j]\subseteq K_{i+j\,({\rm mod\, n)}}$; therefore this is a
$({\Bbb Z} /n{\Bbb Z})$-grading. The Lie ring $K$ also has
the natural ${\Bbb N}$-grading $K=G_1(Y)\oplus G_2(Y)\oplus \cdots $
with respect to the generating set $Y$, where $G_i(Y)$ is the additive subgroup generated by all
commutators of weight $i$ in elements of $Y$.

We define an action of
the Frobenius group $FH$ on $K$ by setting
$k_i^{\varphi}=\omega^i k_i$ for $k_i\in K_i$ and extending this
action to $K$ by linearity.
An action of $H$ is defined on the generating set $Y$
as a cyclic permutation of elements in each subset $O(y_{i_s})$ by the rule
$(y_{r^ki_s})^h=y_{r^{k+1}i_s}$ for $k=0,\ldots, q-2$ and $(y_{r^{q-1}i_s})^h=y_{i_s}$.
Then $O(y_{i_s})$ becomes the $H$-orbit
of the element~$y_{i_s}$. Clearly, $H$ permutes the components $K_i$ by the rule
$K_i^h = K_{ri}$ for all $i\in \Bbb Z/n\Bbb Z$.

Let $J={}_{{\rm id}}\!\left< K_0 \right>$ be the ideal
generated by the $\varphi$-component $K_0$. Clearly,
the ideal $J$ consists of linear combinations of commutators
in elements of $Y$ each of which contains a subcommutator with
zero sum of indices modulo $n$. The
ideal $J$
is generated by homogeneous elements with respect to the gradings
$K=\bigoplus_i G_i(Y)$ and $K=\bigoplus_{i=0}^{n-1} K_i$ and therefore
is homogeneous with respect to both gradings.
Note also that the ideal $J$ is obviously $FH$-invariant.

Let $I={}_{{\rm id}}\!\left< \gamma_{c+1}(C_K(H)) \right>^F$ be
the smallest $F$-invariant ideal containing the subring
$\gamma_{c+1}(C_K(H))$. The ideal $I$ is obviously homogeneous
with respect to the grading $K=\bigoplus_i G_i(Y)$ and is
$FH$-invariant. The fact that the ideal $I$ is $F$-invariant,
implies that $nI\subseteq I_0\oplus \dots \oplus I_{n-1}$, where
$I_k=I\cap K_k$ for $k=0,1,\dots ,n-1$. Indeed, for $z\in I$, for
every $i=0,\ldots, n-1$ we have $z_i:= \sum_{s=0}^{n-1}
\omega^{-is}z^{\varphi^s}\in K_i$ and $nz=\sum_{i=0}^{n-1}z_i$. We
denote $\hat I=I_0\oplus \dots \oplus I_{n-1}$. This is an ideal
of $K$, which is homogeneous with respect to both  gradings
$K=\bigoplus_i G_i(Y)$ and $K=\bigoplus_{i=0}^{n-1} K_i$. It is
also $FH$-invariant, since $I$ is $FH$-invariant and the
components $K_i$ are permuted by $FH$.

Consider the quotient Lie ring $N=K/(J+\hat I)$. Since the ideals $J$ and $\hat I$
are homogeneous with respect to the gradings $K=\bigoplus_i G_i(Y)$ and
$K=\bigoplus_{i=0}^{n-1} K_i$, the quotient ring $N$ has
the corresponding induced gradings.
We use indices to denote the components $N_i$
of the $({\Bbb Z} /n{\Bbb Z})$-grading induced by $K=\bigoplus_{i=0}^{n-1} K_i$.
Note that $N_0=0$ by the construction of $J$.

The group $H$ permutes the grading components of $N=N_1\oplus\dots \oplus N_{n-1}$ with regular orbits of length $q$.
Therefore elements of $C_N(H)$ have the form $a+a^h+\dots +a^{h^{q-1}}$. Hence $C_N(H)$ is contained
in the image of $C_K(H)$ in $N=K/(J+\hat I)$ and therefore $\gamma_{c+1}(C_N(H)) $ is contained in
the image of
the ideal $I$ by its construction. Then $n\gamma_{c+1}(C_N(H))=0 $, since $nI\subseteq \hat I$.

The group $H$ also permutes the $({\Bbb Z} /n{\Bbb Z})$-grading components of $M:
=nN=\bigoplus _{i=0}^{n-1}M_i$, where $M_i=nN_i$,
 with regular orbits of length $q$. Therefore, $C_M(H)=nC_N(H)$ and  $\gamma_{c+1}(C_M(H))=
 \gamma_{c+1}(nC_N(H))=n^{c+1} \gamma_{c+1}(C_N(H))=0 $.

Since $N_0=0$, we also have $M_0=0$.

By Theorem~\ref{kh-ma-shu10-1} for some $(c,q)$-bounded function
$u=u(c,q)$  the Lie ring $n^{u}M$ is nilpotent of $(c,q)$-bounded
class $f-1=f(c,q)-1$. Consequently,
$$
n^{(u+1)f}[y_{i_1}, y_{i_2},\ldots, y_{i_{f}} ]=[n^{u+1}y_{i_1},
n^{u+1}y_{i_2},\ldots, n^{u+1}y_{i_{f}} ]\in J+\hat I.
$$
Note that we should take  the factors $n^{u+1}$ because the
elements $y_{i_s}\in K$ may %!!
not belong to the preimage of $M=nN$. %!!OK
 Since both ideals $J$ and $\hat I$
are homogeneous with respect to the grading $K=\bigoplus_i
G_i(Y)$, this means that the left-hand side is equal modulo the
ideal $\hat I$ to a linear combination of commutators of the same
weight $f$ in elements of $Y$ each of which contains a
subcommutator with zero sum of indices modulo $n$.

Now suppose that $L$ is an arbitrary Lie ring
 satisfying
the hypothesis of Proposition~\ref{combinatorial}, and let $\widetilde L =
L\otimes _{{\Bbb Z} }{\Bbb Z} [\omega ]$. Let $x_{i_1},
x_{i_2},\ldots, x_{i_{f}} $ be arbitrary $\varphi$-homo\-ge\-neous
elements of
$\widetilde L$.
We define the homomorphism $\delta$ from the free Lie ring $K$
into $\widetilde L$ extending the mapping
$$
y_{r^ki_s}\to x_{i_s}^{h^k}\quad \text{for} \quad
s=1,\ldots,f \quad \text{and}\quad k=0,1,\ldots,q-1.
$$
It is easy to see that $\delta$ commutes with the action of $FH$ on $K$ and $\widetilde L$. Therefore
$\delta (O(y_{i_s}))=O(x_{i_s})$ and
$\delta (I)=0$, since $\gamma_{c+1}(C_{\widetilde L} (H)) =0$ and $\delta(C_K(H))\subseteq C_{\widetilde L}(H)$.
We now apply $\delta$ to the representation
 of  $n^{(u+1)f} [y_{i_1}, y_{i_2},\ldots, y_{i_{f}}]$ constructed above. Since $\delta (\hat I)\subseteq \delta(I)=0$,
 as the image we
 obtain a required
 representation of  $n^{(u+1)f}[x_{i_1},x_{i_2},\ldots, x_{i_{f}} ]$ as a linear
combination of commutators of weight $f$
in elements of the set $\delta (Y)=\bigcup_{s=1}^f O(x_{i_s})$
each of which
has a subcommutator with zero sum of indices
modulo $n$.
\end{proof}

\section{Nilpotency class}

We begin with two lemmas that are  well-known in folklore. Induced
automorphisms of invariant subgroups and sections are denoted by
the same letters. Fixed-point subgroups are denoted as
centralizers in the natural semidirect products.

\begin{lemma}[see, e.~g., {\cite[Theorem 1.6.1]{kh4}}]\label{l1} If $\alpha$ is
an automorphism of a finite group $G$ and $N$ is an
$\alpha$-invariant subgroup of $G$, then $|C_{G/N}(\alpha )|\leq
|C_G(\alpha )|$. \qed \end{lemma}

\begin{lemma}[see, e.~g., {\cite[Theorem 1.6.2]{kh4}}]\label{l2-coprime} %OK
If $\alpha$ is an automorphism of a finite group $G$ and $N$ is an
$\alpha$-invariant subgroup of $G$ such that $(|N|, |\alpha|)=1$, %zapiataja
 then
$C_{G/N}(\alpha )= C_G(\alpha )N/N$. \qed \end{lemma}

\begin{lemma}[see, e.~g., {\cite[Corollary 1.7.4]{kh4}}]\label{l-b-r} If
$\varphi $ is an automorphism of order $p^k$ of a finite abelian
$p$-group $A$ and $|C_A(\varphi  )|=p^s$, then the rank of $A$ is at
most $sp^k$. \end{lemma}

The following lemma is a well-known consequence of the theory of powerful $p$-groups \cite{lu-ma}.

\begin{lemma}[see, e.~g., {\cite[Corollary 11.21]{khu-b2}}]\label{l-r-o} If
a finite $p$-group $P$ has rank $r$ and exponent $p^e$, then $|P|$ is $(p,r,e)$-bounded.
\end{lemma}

\begin{proof}[Proof of Theorem~\ref{t1}] Recall that $P$ is a finite $p$-group admitting a
Frobenius group $FH$  of automorphisms with cyclic kernel $F=\langle \varphi \rangle$ of
order $p^k$ and complement $H$ of order $q$. Let $p^m=|C_P(F)|$ and let $C_P(H)$ be  nilpotent of class $c$.
We need to find a characteristic subgroup of $(p,k,m,c)$-bounded index and of $(c,q)$-bounded nilpotency class.

Consider the associated Lie ring $L(P)=\bigoplus _i \gamma _i(P)/\gamma
_{i+1}(P)$, where $\gamma _i$ denotes the $i$th term of the lower
central series (see, e.~g., \S\,3.2 in \cite{kh4}). Extend the
ground ring by a $p^k$-th primitive root of unity $\omega$ setting $L=
L(P)\otimes _{{\Bbb Z} }{\Bbb Z} [\omega ]$ and regarding $L(P)$
as $L(P)\otimes 1$. The group $FH$ naturally acts on $L$. We
define the $\varphi $-components as in \S\,\ref{s-l} (with $n=p^k$);
recall that $p^kL\subseteq L_0+L_1+\dots +L_{p^k-1}$. Since any
$\varphi $-homo\-ge\-neous commutator with zero sum of indices modulo
$p^k$ belongs to $L_0$, by Proposition~\ref{combinatorial} we
obtain
$$
p^{k(f+w)}\gamma _f(L)=p^{kw}\gamma _f(p^kL)\subseteq p^{kw}\gamma
_f(L_0+L_1+\dots +L_{p^k-1})\subseteq {}_{\rm id}\langle
L_0\rangle
$$
for the functions $f=f(c,q)$, $w=w(c,q)$ in that proposition.
Since $L_0=C_{L(P)}(\varphi)\otimes _{{\Bbb Z} }{\Bbb Z} [\omega
]$  and $p^mC_{L(P)}(\varphi)=0$  by Lemma~\ref{l1} and the
Lagrange theorem, we obtain
$$
p^{k(f+w)+m}\gamma _f(L)\subseteq p^m{}_{\rm id}\langle L_0\rangle =0.
$$
In particular, $p^{k(f+w)+m}\gamma _f(L(P)) =0$. In terms of the group
$P$ this means that the factors $\gamma _i(P)/\gamma _{i+1}(P)$ have
exponent dividing $p^{k(f+w)+m}$ for all $i\geq f$.

By Lemmas~\ref{l1}  and \ref{l-b-r}, the rank of every factor $\gamma _i(P)/\gamma _{i+1}(P)$ is at most $mp^k$.
Together with the bound for the exponent, this gives a  bound for the order, which we state as a lemma.

\begin{lemma} \label{ner} Suppose that $P$ is a finite $p$-group admitting a
Frobenius group $FH$  of automorphisms with cyclic kernel
$F=\langle \varphi \rangle$ of order $p^k$ and complement $H$ of order
$q$. Let $p^m=|C_P(F)|$ and let $C_P(H)$ be  nilpotent of class
$c$. Then $|\gamma _i(P)/\gamma _{i+1}(P)|\leq p^{(kf+kw+m)mp^k}$ for all
$i\geq f$, where $f=f(c,q)$ and $w=w(c,q)$ are the functions in
Proposition~\ref{combinatorial}. \end{lemma}

Lemma~\ref{ner} can be applied to any $FH$-invariant subgroup $Q$
of $P$. In particular, we choose $Q=\gamma _{U+1}(P\langle \varphi
\rangle)$, where $U=(kf+kw+m)mp^k$. Clearly, $Q\leq P$, so that
$|C_Q(\varphi  )|\leq p^m$. By Lemma~\ref{ner}, $|\gamma _i(Q)/\gamma
_{i+1}(Q)|\leq p^{U}$  for all $i\geq f$. On the other hand, by
the well-known theorem of P.~Hall \cite[Theorem~2.56]{hall} we
have $|\gamma _i(Q)/\gamma _{i+1}(Q)|\geq p^{U+1}$  if $\gamma _{i+1}(Q)\ne
1$. To avoid a contradiction we must conclude that $\gamma
_{f+1}(Q)=1$. Thus, $Q$ is nilpotent of $(c,q)$-bounded class. %OK, pust' bez f

The automorphism $\varphi $ acts trivially on the factors of the lower
central series of $P\langle \varphi  \rangle$. Since $|C_{P\langle \varphi
\rangle}(\varphi  )|=p^{m+k}$, by Lemma~\ref{l1} the orders of all
these factors are at most $p^{m+k}$. Since the quotient
$P\langle \varphi  \rangle /Q$ is nilpotent of class $U$ by
construction, its order is at most
$p^{(m+k)U}=p^{(m+k)(kf+kw+m)mp^k}$, which is a
$(p,k,m,c)$-bounded number. Thus, $Q$ has  $(p,k,m,c)$-bounded
index in $P$ and $(c,q)$-bounded nilpotency class. The subgroup
$Q$ contains a characteristic subgroup $P^{p^e}$ for some
 $(p,k,m,c)$-bounded number $e$. Since the rank of
 $P$ is  $(p,k,m,c)$-bounded, the index of $P^{p^e}$ in $P$ is also $(p,k,m,c)$-bounded  by Lemma~\ref{l-r-o}.
\end{proof}

We now produce  an example showing that the condition of the kernel being cyclic in Theorem~\ref{t1} is essential.

\begin{example}\label{example}
Let $L$ be a Lie ring whose additive group is the direct sum of
three copies of ${\Bbb Z} _2$, the group of $2$-adic integers, with
generators $e_1,e_2,e_3$ as a ${\Bbb Z} _2$-module, and let the
structure constants of $L$ be $[e_1,e_2]=4e_3$, \
$[e_2,e_3]=4e_1$, \ $[e_3,e_1]=4e_2$. A Frobenius group $FH$ of
order $12$ acts on $L$ as follows: $F=\{1,f_1,f_2,f_3\}$, where
$f_i(e_i)=e_i$ and $f_i(e_j)=-e_j$ for $i\ne j$, and $H=\langle
h\rangle$ with $h(e_i)=e_{i+1 \,({\rm mod}\,3)}$. Since $L$ is a
powerful Lie ${\Bbb Z} _2$-algebra, by \cite[Theorem~9.8]{p-ad-anal} the
Baker--Campbell--Hausdorff formula defines the structure of a
uniformly powerful pro-$2$-group $P$ on the same set $L$. For any
positive integer $n$, the quotient of $P$ by $P^{2^n}=2^nL$ is a
finite $2$-group $T$. The induced action of $FH$ on $T$ is such
that $|C_T(F)|=8$ and $C_T(H)$ is cyclic, while the derived length
of $T$ is about $\log _4 n$.
\end{example}

\section{Order, rank, and exponent}
Suppose that a finite abelian group $V$ admits  a Frobenius
group of automorphisms $FH$ with cyclic kernel $F=\langle\varphi \rangle $ of order $n$.
We can extend the ground ring by a primitive $n$-th root of unity $\omega$
forming $W=V\otimes _{{\Bbb Z}}{\Bbb Z} [\omega]$ and define the natural action of the group $FH$ on $W$. As a ${\Bbb Z}$-module (abelian group),
${\Bbb Z} [\omega ]=\bigoplus _{i=0}^{E(n)-1} \omega ^i {\Bbb Z}$, where $E(n)$ is the Euler
function. Hence,
 \begin{equation} \label{euler0}
 W=\bigoplus _{i=0}^{E(n)-1} V\otimes \omega ^i {\Bbb Z},
 \end{equation}
 so that $|W|=|V|^{E(n)}$.  Similarly,
$C_W(\varphi )= \bigoplus _{i=0}^{E(n)-1} C_{V}(\varphi )\otimes \omega ^i {\Bbb Z}$, so that
$|C_W(\varphi )|=|C_{V}(\varphi )|^{E(n)}$.

As in \S\,\ref{s-l}
 for $\widetilde L$, we define $\varphi$-\textit{components} $W_k$ for
$k=0,\,1,\,\ldots ,n-1$ as the `eigensubspaces'
$$W_k=\left\{ a\in W\mid a^{\varphi}=\omega ^{k}a\right\} .$$
Recall that  $W$ is an `almost direct sum' of the $W_i$:
namely,
\begin{equation}\label{eq11}
nW\subseteq W_0+W_1+\dots + W_{n-1}
\end{equation}
and
\begin{equation}\label{eq12}
\text{if}\quad  w_0+w_1+\dots + w_{n-1}=0 \quad\text{for }
w_i\in W_i, \quad \text{then}\quad nw_i=0\quad \text{for all }
i.
\end{equation}
As in \S\,2 we refer to elements  of
$\varphi$-components as being \textit{$\varphi$-homo\-ge\-neous}, and
apply the Index Convention using lower indices of small Latin
letters to only indicate the $\varphi $-component containing this
element.

As before, since the kernel $F$ of the Frobenius group $FH$ is
cyclic, the complement $H$ is also cyclic,  $H= \langle h
\rangle$, say, of order $q$, and $\varphi^{h^{-1}} = \varphi^{r}$
for some $1\leq r \leq n-1$, which is a primitive $q$-th root of
unity in ${\Bbb Z}/n {\Bbb Z}$. The group $H$ permutes the
$\varphi$-components $W_i$ by the rule $W_i^h = W_{ri}$ for
all $i\in \Bbb Z/n\Bbb Z$. For $u_k\in W_k$ we denote $u_k^{h^i}$
by $u_{r^ik}$ under the Index Convention.

From now on we assume in addition that $V$ is an abelian $FH$-invariant section of
the $p$-group $P$ in Theorem~\ref{t-g}. Recall that $|\varphi  |=n=p^k$ and $|C_P(\varphi  )|=p^m$.

\begin{lemma}\label{l-ab}
There is a characteristic subgroup $U$ of $V$ such that $|U|$ is $(p,k,m)$-bounded and

{\rm (a)} $|V/U|\leq |C_V(H)|^{|H|}$;

{\rm (b)} the rank of $V/U$ is at most $r|H|$, where  $r$ is the rank of $C_P(H)$;

{\rm (c)} the exponent of $V/U$ is at most $p^{e}$, where $p^e$ is the exponent of $C_P(H)$.
\end{lemma}

\begin{proof}
 The group $H$ acts on the set of $\varphi $-components $W_i$ with one single-element   orbit $\{W_0\}$ and
$(p^k-1)/q$ regular orbits. We choose one element in every regular $H$-orbit
 and let $Y=\sum _{j=1}^{(p^k-1)/q}W_{i_j}$ be the sum of these chosen $\varphi $-components. The mapping
 $\vartheta : y\to y+y^h+\dots +y^{h^{q-1}}$ is a homomorphism of the abelian group $Y$ into $C_W(H)$. We claim that
$p^k{\rm Ker}\,\vartheta =0$. Indeed, if $y\in {\rm
Ker}\,\vartheta$ is written as $y=\sum _{j=1}^{(p^k-1)/q}y_{i_j}$
for $y_{i_j}\in W_{i_j}$, then $\vartheta (y)$ is equal to $y$
plus a linear combination of elements of $\varphi $-components
$W_{r^li_j}$ with all the indices $r^li_j$ being different from
the indices  $i_1,\dots ,i_{(p^k-1)/q}$. Therefore the equation
$\vartheta (y)=0$ implies $p^ky_{i_j}=0$ by \eqref{eq12}, so that
$p^ky=0$. Clearly, $|Y/{\rm Ker}\,\vartheta |\leq |C_W(H)|$, the
rank of $Y/{\rm Ker}\,\vartheta $ is at most the rank of $C_W(H)$,
and  the exponent of $Y/{\rm Ker}\,\vartheta $ is at most the
exponent of $C_W(H)$.

Let $p^f$ be the maximum of $p^k$ and the exponent of $W_0$,
which is a $(p,k,m)$-bounded number. Then $\Omega _f(W)\geq
W_0+{\rm Ker}\,\vartheta$ (where we use the standard notation $\Omega _i$ for the subgroup generated by all elements of order dividing $p^i$). Since
$$
p^kW\leq W_0+W_1+\dots
+W_{p^k-1}=W_0+Y+Y^h+\dots +Y^{h^{q-1}},
$$
 we obtain the
following.

\begin{lemma} \label{incl}
The image of  $p^kW$ in $W/\Omega _{f}(W)$
is contained in the image of $Y+Y^h+\dots +Y^{h^{q-1}}$ in $W/\Omega _{f}(W)$,
and the image of $Y$ is a homomorphic image of $Y/{\rm Ker}\,\vartheta $.
\end{lemma}

We claim that $U=\Omega _{f+k} (V)$ is the required characteristic subgroup.
The rank of the abelian group $V$ is at most $mp^k$ by
Lemmas~\ref{l1} and~\ref{l-b-r}.
Hence $\Omega _{f+k} (V)$ being of bounded exponent
 has  $(p,k,m)$-bounded order. We now verify that parts (a), (b), (c) are satisfied.

(a) In the abelian $p$-group $W$ the order of the image of  $p^kW$
in $W/\Omega _{f}(W)$ is equal to $|W/\Omega _{f+k}(W)|$.
Therefore Lemma~\ref{incl} and the fact that $|Y/{\rm
Ker}\,\vartheta|\leq |C_W(H)|$ imply %OK
 \begin{equation}\label{ner-ord} |W/\Omega
_{f+k}|\leq  |Y/{\rm Ker}\,\vartheta |^{|H|}\leq |C_W(H)|^{|H|}.
\end{equation}
 Clearly, $\Omega _{f+k} (W)=\Omega _{f+k} (V)\otimes _{{\Bbb Z}}{\Bbb Z}
[\omega ]$  and therefore $|\Omega _{f+k} (W)|=|\Omega _{f+k}
(V)|^{E(p^k)}$. Since $|W|=|V|^{E(p^k)}$ and $|C_W(\varphi
)|=|C_{V}(\varphi )|^{E(p^k)}$, taking the $E(p^k)$-th root of
both sides of \eqref{ner-ord} gives $|V/\Omega _{f+k}(V)|\leq
|C_V(H)|^{|H|}$.

(b) Similarly,  the rank of the image of  $p^kW$ in $W/\Omega
_{f}(W)$ is equal to the rank of $W/\Omega _{f+k}$. By
Lemma~\ref{incl} we obtain that the rank of $W/\Omega _{f+k}(W)$
is at most $|H|$ times the rank of $Y/{\rm Ker}\,\vartheta$, which %!!zapiat
in turn
 is at most the rank of $C_W(H)$). %OK
 Since the ranks are
multiplied by $E(p^k)$ when passing from $V$ to $W$, we obtain
that  the rank of $V/\Omega _{f+k}(V)$ is at most $|H|$ times the
rank of $C_V(H)$, which in turn does not exceed $r$, the rank  of
$C_P(H)$, because $C_P(H)$ covers $C_V(H)$ by Lemma
\ref{l2-coprime} since the action of $H$ is coprime.

(c) Finally, the exponent of the image of  $p^kW$ in $W/\Omega
_{f}(W)$ is equal to the exponent of $W/\Omega _{f+k}$. By
Lemma~\ref{incl} we obtain that  the exponent of $W/\Omega
_{f+k}(W)$ does not exceed  the exponent of $Y/{\rm
Ker}\,\vartheta$ which is at most the exponent of $C_W(H)$ and,
consequently, that of  $C_V(H)$.  Since the action of $H$ is
coprime,  by Lemma \ref{l2-coprime} the exponent of $C_V(H)$ (and
therefore the exponent of $W/\Omega _{f+k}(W)$ as well)  is at
most
$p^e$, the exponent of $C_P(H)$. %OK
\end{proof}

\begin{proof}[Proof of Theorem~\ref{t-g}]
Recall that  $P$ is a finite $p$-group  admitting a Frobenius group $FH$  of automorphisms with cyclic kernel $F$ of
order $p^k$ with  $p^m=|C_P(F)|$ fixed points of the kernel. Let $p^s=|C_P(H)|$, let $r$ be the rank of $C_P(H)$,
and $p^e$ the exponent of $C_P(H)$. We need to find a characteristic subgroup $Q$ of $(p,k,m)$-bounded index
with required bounds for the order, rank, and exponent.
We can of course find such a subgroup separately for each of these parameters and then take the intersection.

By Lemmas~\ref{l1}  and \ref{l-b-r}, the rank of $P$ is at most
$mp^k$. Hence $P$ has a characteristic powerful subgroup of
$(p,k,m)$-bounded index by \cite[Theorem~1.14]{lu-ma}. Therefore
we can assume $P$ to be powerful from the outset.

By \cite{khu93}  (see also \cite[Theorem~12.15]{khu-b2}),
 the group $P$ has a characteristic subgroup
$P_1$ of $(p,k,m)$-bounded index that is soluble of $p^k$-bounded derived length
at most $2K(p^k)$
(where $K$ is Kreknin's function bounding the derived length of a Lie ring with a fixed-point-free
 automorphism of order $p^k$). Let $\mathcal{D}$ be the set of %!!
  factors of the derived series of $P_1$.
  For any $V\in \mathcal{D} $,  we have, by Lemma~\ref{l-ab},  that %OK
  $|V|\leq p^{g}|C_V(H)|^{|H|}$ for some
$(p,k,m)$-bounded number $g=g(p,k,m)$. Then
 $$
 |P_1|= \prod _{V\in \mathcal{D}} |V| \leq p^{2gK(p^k)}\prod _{V\in \mathcal{D}} |C_V(H)|^{|H|}=p^{2gK(p^k)} |C_{P_1}(H)|^{|H|}
 $$
  by Lemma
\ref{l2-coprime}, %!!
since the action
of $H$ is coprime. Since the rank of the powerful $p$-group $P$ is
at most $mp^k$, by taking the $(p,k,m)$-bounded power
$P_1^{f(p,k,m)}$ with $f(p,k,m)= p^{2gK(p^k)}$ we obtain a
characteristic subgroup which has $(p,k,m)$-bounded index by
Lemma~\ref{l-r-o}. The  order of $P_1^{f(p,k,m)}$ is at most
$|C_{P}(H)|^{|H|}$. Indeed, either the exponent of $P_1$ is at
most $f(p,k,m)$ and then $P_1^{f(p,k,m)}=1$, %!!zapiat.
 or the exponent of
$P_1$ is greater than $f(p,k,m)$ and then  $|P_1:
P_1^{f(p,k,m)}|\geq f(p,q,m)$, %!! zapiat.
 whence $|P_1^{f(p,k,m)}|\leq
|C_{P_1}(H)|^{|H|}\leq |C_{P}(H)|^{|H|}$. %OK

The powerful $p$-group $P$ has a series
\begin{equation}\label{unif}
P>P^{p^{k_1}}>P^{p^{k_2}}>\cdots >1
\end{equation}
with uniformly powerful
factors of strictly decreasing ranks. For every factor $S$ of this
series having exponent, say, $p^t$, its  subgroup
$V=S^{p^{[(t+1)/2]}}$ is abelian. By Lemma~\ref{l-ab} the subgroup
$V$ has a characteristic subgroup $U$ of $(p,k,m)$-bounded order
such that the rank of $V/U$ is at most $r|H|$. The rank of $V$ is
equal to the rank of $S$ and $V$ is generated by elements of order
$p^{[t/2]}$. If the rank of $S$ is higher than the %!!
rank of $U$, then %!!
there exists an element of order $p^{[t/2]}$ that belongs to
$U$ and thus $t$ should be $(p,k,m)$-bounded.  %OK
 Therefore the rank
of $S$ can be higher than $r|H|$ only if the exponent of $S$ is
$(p,k,m)$-bounded. Since the rank of $P$ is at most $mp^k$, all
the factors in \eqref{unif} of rank higher than $r|H|$ combine in
a quotient $P/ P^{p^{k_u}}$ of $(p,k,m)$-bounded order; then
$P^{p^{k_u}}$ is the required characteristic subgroup of
$(p,k,m)$-bounded index and of rank at most $r|H|$.

Let $p^v$ be the exponent of $P$. Since in the powerful group $P$
the series $P>P^{p}\geq P^{p^2}\geq P^{p^3}\geq \cdots {}$ is
central, the subgroup $P^{p^{[(v+1)/2]}}$ is abelian. By
Lemma~\ref{l-ab} the exponent of $P^{p^{[(v+1)/2]}}$ is at most
$p^{e+f}$ for some $(p,k,m)$-bounded number $f$. Hence the
exponent of $P$ is at most $p^{2e+g}$ for some $(p,k,m)$-bounded
number $g=g(p,k,m)$. Since the rank of $P$ is at most $mp^k$, by
Lemma~\ref{l-r-o} the characteristic subgroup $P^{p^g}$ has
$(p,k,m)$-bounded index and exponent at most  $p^{2e}$. \end{proof}

\end{document}